\documentclass{ws-m3as-ppn}

\usepackage[colorlinks=true]{hyperref}
\hypersetup{urlcolor=blue, citecolor=red, linkcolor=blue}

\newcommand{\be}[1]{\begin{equation}\label{#1}}
\newcommand{\ee}{\end{equation}}
\renewcommand{\(}{\left(}
\renewcommand{\)}{\right)}
\newcommand{\R}{{\mathbb R}}
\newcommand{\N}{{\mathbb N}}

\newcommand{\ird}[1]{\int_{\R^d}{#1}\,dx}

\newcommand{\iri}[1]{\int_{\R^d}{#1}\,d\mu}
\newcommand{\M}[1]{\mathsf M_{#1}}
\renewcommand{\H}[1]{\mathsf H_{#1}}

\newcommand{\nc}{\normalcolor}

\newcommand{\finprf}{\unskip\null\hfill$\square$\vskip 0.3cm}

\renewcommand{\(}{\left(}
\renewcommand{\)}{\right)}
\renewcommand{\email}[1]{{\small E-mail:{\textsf{#1}}}}

\newcommand{\iomuone}[1]{\int_{\R^{d_1}} #1\,d\gamma_1}
\newcommand{\iomutwo}[1]{\int_{\R^{d_2}} #1\,d\gamma_2}
\newcommand{\iomuonetwo}[1]{\iint_{\R^{d_1}\times\R^{d_2}} #1\,d\gamma_1\otimes\gamma_2}

\newcommand{\nx}{\nabla_x}

\newcommand{\idg}[1]{\int_{\R^d}{#1}\,d\gamma}
\newcommand{\idgb}[1]{\int_{\Omega}{#1}\,d\gamma}
\newcommand{\potential}{\psi}
\newcommand{\nnrm}[2]{\left\|{#1}\right\|_{\mathrm L^{#2}(\R^d\times\R^d)}}
\newcommand{\nrmg}[2]{\left\|{#1}\right\|_{\mathrm L^{#2}(\R^d,d\gamma)}}
\newcommand{\nnrmu}[2]{\left\|{#1}\right\|_{\mathrm L^{#2}(\R^d\times\R^d,d\mu)}}

\newcommand{\iirdmu}[1]{\iint_{\R^d\times\R^d}{#1}\,d\mu}

\newcommand{\munu}{\nu}

\begin{document}

\markboth{J.~Dolbeault \& X.~Li}{$\varphi$-entropies for Fokker-Planck and kinetic Fokker-Planck equations}

\title{$\varphi$-entropies: convexity, coercivity and hypocoercivity\\
for Fokker-Planck and kinetic Fokker-Planck equations}

\author{Jean DOLBEAULT}

\address{CEREMADE (CNRS UMR n$^\circ$ 7534),\\
PSL research university, Universit\'e Paris-Dauphine,\\
Place de Lattre de Tassigny, 75775 Paris 16, France\\
\email{ dolbeaul@ceremade.dauphine.fr}}

\author{Xingyu LI}

\address{CEREMADE (CNRS UMR n$^\circ$ 7534),\\
PSL research university, Universit\'e Paris-Dauphine,\\
Place de Lattre de Tassigny, 75775 Paris 16, France\\
\email{ li@ceremade.dauphine.fr}}

\maketitle
\thispagestyle{empty}

\date{\today}\begin{center}\small\emph{\today}\end{center}

\begin{abstract} This paper is devoted to $\varphi$-entropies applied to Fokker-Planck and kinetic Fokker-Planck equations in the whole space, with confinement. The so-called $\varphi$-entropies are Lyapunov functionals which typically interpolate between Gibbs entropies and $\mathrm L^2$ estimates. We review some of their properties in the case of diffusion equations of Fokker-Planck type, give new and simplified proofs, and then adapt these methods to a kinetic Fokker-Planck equation acting on a phase space with positions and velocities. At kinetic level, since the diffusion only acts on the velocity variable, the transport operator plays an essential role in the relaxation process. Here we adopt the $\mathrm H^1$ point of view and establish a sharp decay rate. Rather than giving general but quantitatively vague estimates, our goal here is to consider simple cases, benchmark available methods and obtain sharp estimates on a key example. Some $\varphi$-entropies give rise to improved entropy -- entropy production inequalities and, as a consequence, to faster decay rates for entropy estimates of solutions to non-degenerate diffusion equations. We prove that faster entropy decay also holds at kinetic level away from equilibrium and that optimal decay rates are achieved only in asymptotic regimes. \end{abstract}


\section{Introduction}\label{Sec:Intro}

By definition, the $\varphi$\emph{-entropy} of a nonnegative function $w\in\mathrm L^1(\R^d,d\gamma)$ is the functional
\[
\mathcal E[w]:=\idg{\varphi(w)}\,,
\]
where $\varphi$ is a nonnegative convex continuous function on $\R^+$ such that $\varphi(1)=0$ and $1/\varphi''$ is concave on $(0,+\infty)$, \emph{i.e.},
\be{Ass:Convexity}
\varphi''\ge0\,,\quad\varphi\ge\varphi(1)=0\quad\mbox{and}\quad(1/\varphi'')''\le0\,.
\ee
Notice that the last condition means $2\,(\varphi''')^2\le\varphi''\,\varphi^{(iv)}$ a.e. A classical example of such a function $\varphi$ is given by
\[
\varphi_p(w):=\tfrac1{p-1}\,\big(w^p-1-p\,(w-1)\big)\quad p\in(1,2]
\]
where, in the case $p=2$, $\varphi_2(w)=(w-1)^2$ and the limit case as $p\to1_+$ is given by the standard Gibbs entropy
\[
\varphi_1(w):=w\,\log w-(w-1)\,.
\]
Many results corresponding to the case $p=2$ can be obtained, \emph{e.g.}, by spectral methods. The case $p=1$ is important in probability theory and statistical physics. Our goal is to emphasize that they share properties which can be put in a common framework. Throughout this paper we shall assume that $d\gamma$ is a nonnegative bounded measure, which is absolutely continuous with respect to Lebesgue's measure and write
\[
d\gamma=e^{-\potential}\,dx
\]
where $\potential$ is a \emph{potential} such that $e^{-\potential}$ is in $\mathrm L^1(\R^d,dx)$. Up to the addition of a constant to $\potential$, we can assume without loss of generality that $d\gamma$ is a probability measure. A review of the \emph{main properties of $\varphi$-entropies}, new and simplified proofs and key references are given in Section~\ref{Sec:PhiEntropy}.

\medskip Without entering the technical details, let us illustrate the use of the $\varphi$-entropy in the case of diffusion equations. A typical application of the $\varphi$-entropy is the control of the rate of relaxation of the solution to the \emph{Ornstein-Uhlenbeck equation}
\be{Eqn:OU}
\frac{\partial w}{\partial t}=\mathsf L\,w:=\Delta w-\nabla\potential\cdot\nabla w\,,
\ee
which is also known as the \emph{backward Kolmogorov equation}. If we solve the equation with a nonnegative initial datum $w_0$ such that $\idg{w_0}=1$, then the solution satisfies $\idg{w(t,\cdot)}=1$ for any $t>0$ and $\lim_{t\to+\infty}w(t,\cdot)=1$. The Ornstein-Uhlenbeck operator $\mathsf L$ defined on $\mathrm L^2(\R^d,d\gamma)$ is indeed self-adjoint and such that
\[
-\idg{(\mathsf L\,w_1)\,w_2}=\idg{\nabla w_1\cdot\nabla w_2}\quad\forall\,w_1,\,w_2\in\mathrm H^1(\R^d,d\gamma)\,.
\]
As a consequence, it is also straightforward to observe that for any solution $w$ with initial datum $w_0$ such that $\mathcal E[w_0]$ is finite, then
\[
\frac d{dt}\mathcal E[w]=-\idg{\varphi''(w)\,|\nx w|^2}=:-\,\mathcal I[w]\,,
\]
where $\mathcal I[w]$ denotes the $\varphi$\emph{-Fisher information} functional. If for some $\Lambda>0$ we can establish the \emph{entropy -- entropy production} inequality
\be{Ineq:E-EP}
\mathcal I[w]\ge\Lambda\,\mathcal E[w]\quad\forall\,w\in\mathrm H^1(\R^d,d\gamma)\,,
\ee
then we deduce that
\[
\mathcal E[w(t,\cdot)]\le\mathcal E[w_0]\,e^{-\Lambda\,t}\quad\forall\,t\ge0\,,
\]
which controls the convergence of $w$ to $1$ as $t\to+\infty$, for instance in $\mathrm L^p(\R^d,d\gamma)$ by a generalized \emph{Csisz\'ar-Kullback inequality} if $\varphi=\varphi_p$, $1\le p\le2$. The entropy -- entropy production inequality is the Poincar\'e inequality associated with $d\gamma$ if $\varphi=\varphi_2$, and the logarithmic Sobolev inequality if $\varphi=\varphi_1$.

We recall that the study of~\eqref{Eqn:OU} is equivalent to the study of the \emph{Fokker-Planck equation}
\be{Eqn:FP}
\frac{\partial u}{\partial t}=\Delta u+\nx\cdot(u\,\nx \potential)\,.
\ee
A nonnegative solution with initial datum $u_0\in\mathrm L^1(\R^d,dx)$ and $\ird{u_0}=M>0$ has constant mass $M=\ird{u(t,\cdot)}$ for any $t>0$, and converges towards the unique stationary solution
\[
u_\star=M\,\frac{e^{-\potential}}{\ird{e^{-\potential}}}\,.
\]
Without loss of generality, we shall assume that $M=1$. Then one observes that $w=u/u_\star$ solves~\eqref{Eqn:OU}, which allows to control the rate of convergence of $u$ to~$u_\star$. A list of results concerning the solutions of~\eqref{Eqn:OU} and~\eqref{Eqn:FP} is also collected in Section~\ref{Sec:PhiEntropy}.

\medskip The third section of this paper is devoted to the extension of $\varphi$-entropy methods to kinetic equations. Section~\ref{Sec:kFPsharp} of this paper deals with the \emph{kinetic Fokker-Planck equation}, or \emph{Vlasov-Fokker-Planck equation}, that can be written as
\be{VFP1}
\frac{\partial f}{\partial t}+v\cdot\nabla_x f-\nabla_x\potential\cdot\nabla_v f=\Delta_v f+\nabla_v\cdot\(v\,f\)\,.
\ee
Our basic example corresponds to the case of the harmonic potential $\potential(x)=|x|^2/2$. Unless it is explicitly specified, we will only consider this case. Notice that this problem has an explicit Green function whose expression can be found in~\cite{MR2042214}.

Since~\eqref{VFP1} is linear, we can assume at no cost that $\nnrm f1=1$ and consider the stationary solution
\[
f_\star(x,v)=(2\,\pi)^{-d}\,e^{-\potential(x)}\,e^{-\tfrac12|v|^2}=(2\,\pi)^{-d}\,e^{-\tfrac12(|x|^2+|v|^2)}\quad\forall\,(x,v)\in\R^d\times\R^d\,.
\]
The function
\[
g:=\frac f{f_\star}
\]
solves
\be{VFP2}
\frac{\partial g}{\partial t}+\mathsf Tg=\mathsf L\,g
\ee
where the transport operator $\mathsf T$ and the Ornstein-Uhlenbeck operator $\mathsf L$ are defined respectively by
\[
\mathsf Tg:=v\cdot\nabla_x g-x\cdot\nabla_v g\quad\mbox{and}\quad\mathsf L\,g:=\Delta_v g-v\cdot\nabla_v g\,.
\]
Let $d\mu:=f_\star\,dx\,dv$ be the invariant measure on the phase space $\R^d\times\R^d$, so that $\mathsf T$ and $\mathsf L$ are respectively anti-self-adjoint and self-adjoint. The function
\[
h:=g^{p/2}
\]
solves
\be{VFP3}
\frac{\partial h}{\partial t}+\mathsf Th=\mathsf L\,h+\frac{2-p}p\,\frac{|\nabla_v h|^2}h\,.
\ee
At the kinetic level, we consider the $\varphi$-entropy given by
\[
\mathcal E[g]:=\iirdmu{\varphi(g)}\,.
\]
With this notation, $\mathcal E[g]=\iirdmu{\varphi\(f/f_\star\)}$ so that, with $f=g\,f_\star=h^{2/p}\,f_\star$ we have
\[
\mathcal E[g]=\iirdmu{h^2\,\log\(\frac{h^2}{\iirdmu{h^2}}\)}\quad\mbox{if}\quad\varphi=\varphi_1\,,
\]
\begin{multline*}
\mathcal E[g]=\mathcal E[h^{2/p}]=\frac1{p-1}\left[\iirdmu{h^2}-\(\iirdmu{h^{2/p}}\)^{p/2}\right]\\
\mbox{if}\quad\varphi=\varphi_p\,,\;p\in(1,2]\,.
\end{multline*}
The optimal rate of decay of $\mathcal E[g]$ has been established by A.~Arnold and J.~Erb in~\cite{2014arXiv1409.5425A}. In the special case of a harmonic potential, their result goes as follows.
\begin{proposition}\label{Prop:AE} Assume that $\potential(x)=|x|^2/2$ for any $x\in\R^d$. Take $\varphi=\varphi_p$ for some $p\in[1,2]$. To any nonnegative solution $g\in\mathrm L^1(\R^d\times\R^d)$ of~\eqref{VFP2} with initial datum $g$ such that $\mathcal E[g_0]<\infty$, we can associate a constant $\mathcal C>0$ for which
\be{ExpDecayVFP}
\mathcal E[g(t,\cdot,\cdot)]\le\mathcal C\,e^{-t}\quad\forall\,t\ge0\,.
\ee
Moreover the rate $e^{-t}$ is sharp as $t\to+\infty$.\end{proposition}
The striking point of this \emph{hypocoercivity} result is to identify the sharp rate of decay. The rate is specific of the harmonic potential $\potential(x)=|x|^2/2$, but it turns out to be useful for the comparison with rates obtained by other methods. Although probably not optimal, a precise estimate of $\mathcal C$ will be given in Section~\ref{Sec:kFPsharp}, with a simplified proof of Proposition~\ref{Prop:AE}.

The method is based on the use of a \emph{Fisher information} type functional
\be{TwistedFisher}
\mathcal J[h]=\tfrac12\iri{|\nabla_vh|^2}+\tfrac12\iri{|\nabla_xh|^2}+\tfrac12\iri{|\nabla_xh+\nabla_vh|^2}
\ee
which involves derivatives in $x$ and $v$. If $h$ solves~\eqref{VFP2}, then the key estimate is to prove that
\[
\frac d{dt}\mathcal J[h(t,\cdot)]\le-\,\mathcal J[h(t,\cdot)]\,.
\]
The result of Proposition~\ref{Prop:AE} follows from the entropy -- entropy production inequality~\eqref{Ineq:E-EPb} that will be established in Proposition~\ref{Prop:BE}: since
\[
\Lambda\,\mathcal E[g(t,\cdot,\cdot)]=\Lambda\,\mathcal E[h^{2/p}]\le\mathcal J[h]\,,
\]
then $\mathcal E[g(t,\cdot,\cdot)]$ has an exponential decay. However, we underline the fact that
\[
\frac d{dt}\mathcal E[g(t,\cdot)]=-\iri{|\nabla_vh|^2}\neq-\,\mathcal J[h(t,\cdot)]\,.
\]

At the level of non-degenerate diffusions, a distinctive property of the $\varphi$\emph{-entropy} with $\varphi=\varphi_p$ and $p\in(1,2)$ is that the entropy -- entropy production inequality
$\mathcal I\ge\Lambda\,\mathcal E$ with an optimal constant $\Lambda>0$ can be improved in the sense that there exists a strictly convex function $F$ on $\R^+$ with $F(0)=0$ and $F'(0)=1$ such that $\mathcal I\ge\Lambda\,F(\mathcal E)$. This has been established in~\cite{MR2152502} and details will be given in Section~\ref{Sec:iEEP}. The key issue is to prove that for some function $\rho$ on $\R^+$, which depends on the solution~$w$, such that $\rho>\Lambda$ a.e., we have $\frac d{dt}\mathcal I[w(t,\cdot)]\le-\,\rho(t)\,\mathcal I[w(t,\cdot)]$. One may wonder if a similar result also holds in the hypocorcive kinetic Fokker-Planck equation. So far, no global improved inequality has been established. What we shall prove is that, if we consider the more general \emph{Fisher information} functional
\be{TwistedFisher2}
\mathcal J_\lambda[h]=(1-\lambda)\iri{|\nabla_vh|^2}+(1-\lambda)\iri{|\nabla_xh|^2}+\lambda\iri{|\nabla_xh+\nabla_vh|^2}\,,
\ee
then for an appropriate choice of $\lambda$ (which turns out to be $t$-dependent), the rate of decay is faster than $e^{-t}$ up to a zero-measure set in $t$. The precise statement, which is our main result, goes as follows.
\begin{theorem}\label{Thm:Main} Let $p\in(1,2)$ and $h$ be a solution of~\eqref{VFP3} with initial datum $h_0\in\mathrm L^1\cap\mathrm L^p(\R^d,d\gamma)$, $h_0\not\equiv1$, and $d\gamma$ be the Gaussian probability measure corresponding to the harmonic potential potential $\potential(x)=|x|^2/2$. Then there exists a function $\lambda:\R^+\to[1/2,1)$ such that $\lambda(0)=\lim_{t\to+\infty}\lambda(t)=1/2$ and a continuous function~$\rho$ on $\R^+$ such that $\rho>1/2$ a.e., for which we have
\[
\frac d{dt}\mathcal J_{\lambda(t)}[h(t,\cdot)]\le-\,2\,\rho(t)\,\mathcal J_{\lambda(t)}[h(t,\cdot)]\quad\forall\,t\ge0\,.
\]
As a consequence, for any $t\ge0$ we have the global estimate
\[
\mathcal J_{\lambda(t)}[h(t,\cdot)]\le\mathcal J_{1/2}[h_0]\,\exp\(-\,2\int_0^t\rho(s)\,ds\)\,.
\]\end{theorem}
This result is weaker than the result for non-degenerate diffusions. The qualitative issues are easy to understand and to some extent classical in the hypocoercivity theory, but no quantitative estimate of $\rho$ in terms of~$h$ is known so far. If $\varphi_p$-entropies were initially thought as interesting objects which interpolate between the Gibbs entropy and standard $\mathrm L^2$ estimates, improved entropy -- entropy production inequalities and the result of Theorem~\ref{Thm:Main} capture an important feature when $p\in(1,2)$: faster rates of decay for finite values of $t$. As $t\to+\infty$, we cannot expect a faster decay rate, but we gain a pre-factor which is less than $1$. See Section~\ref{Sec:conclusion} for more details.

\medskip Let us conclude this introduction with a brief review of the literature. Fokker-Planck equations like~\eqref{Eqn:FP} are ubiquitous in various areas of physics ranging from the description of the motion of particles in a gas or a solute to semi-conductor physics, models of stars in astrophysics or models of populations in biology and social sciences, as microscopic dynamics involving Brownian motion are represented at macroscopic scales by diffusion equations. Second order dynamics (in which forces produce acceleration) in random environments obey in many cases to the Langevin equation and at macroscopic scale the corresponding distribution function solves~\eqref{VFP1}. A typical example is given by particles having random encounters with some background obstacles, a situation that can be encountered in many areas of physical modeling. It has to be emphasized that~\eqref{Eqn:FP} appears in the diffusion limit of the solutions of~\eqref{VFP1}, that is, in the \emph{overdamped regime} in which friction and other forces equilibrate very fast, so that the velocity instantaneously adapts to the forces, which results in first order dynamics. For some general properties of~\eqref{Eqn:FP} and~\eqref{VFP1}, a review of stochastic and PDE methods and some entries to applied cases, we refer for instance to~\cite{MR749386,MR3288096}, among many other books on this topic.

The word ``hypocoercivity'' is apparently due to T.~Gallay and was made popular by C.~Villani in~\cite{MR2275692}. Our computations are based on Villani's ideas in~Section~3 of~\cite{MR2275692} (also see~\cite{MR2562709}), but the use of \emph{twisted gradients} involving simultaneously derivatives in $x$ and $v$ can be also found in~\cite{MR2294477} and in earlier works like~\cite{MR2034753}. It is actually a consequence of H\"ormander's hypoelliptic theory, which covers simultaneously regularization properties and large time behaviour. One can refer for instance to~\cite{MR1969727,MR2034753,MR2130405} and, much earlier, to~\cite{MR0222474}. The seed for such an approach can actually be traced back to Kolmogorov's computation of Green's kernel for the kinetic Fokker-Planck equation in~\cite{MR1503147}, which has been reconsidered by~\cite{il1964equations} from a more modern point of view and successfully applied, for instance, to the study of the Vlasov-Poisson-Fokker-Planck system in~\cite{MR1052014,MR1126388,MR1200643}. In case of the kinetic Fokker-Planck equation, we can refer to~\cite{MR2130405,MR2294477} in the case of a general potential of confinement, and more specifically to~\cite{2014arXiv1409.5425A}. In this last paper, the authors deal with the issue of accurate rates: ``while the main theorem in~\cite{MR2562709} covers a wide class of problems, the price paid is in the estimate for the decay rate, which is off by orders of magnitude.'' The result of Proposition~\ref{Prop:AE} addresses the issue of the optimal rate in a very simple case. For completion, one also has to mention~\cite{doi:10.1063/1.4977475} and~\cite{2017arXiv170203685I} for further theoretical and numerical results.

A twin problem of the kinetic Fokker-Planck equation is the linear BGK model, which has no regularizing properties but shares many common features with the kinetic Fokker-Planck equation as soon as we are concerned with rates of convergence. We refer to~\cite{Herau,Mouhot-Neumann} for early contributions, to~\cite{Dolbeault2009511,MR3324910,2017arXiv170806180B,2017arXiv170310504M,MR3557714} for more recent ones, and especially to~\cite{2017arXiv170204168E}. In this last paper, J.~Evans studies the linear BGK model and a kinetic Fokker-Planck equation on the torus using the $\varphi$-entropies.

In~\cite{MR2275692}, only the cases $p=1$ and $p=2$ were considered, but it is well known since the founding work~\cite{Bakry-Emery85} of Bakry and Emery that intermediate values of $p$ can then be considered. In the case of $\varphi$-entropies associated with non-degenerate diffusions, this idea was invoked on many occasions, for instance in~\cite{MR954373,MR1796718,MR2081075,Arnold-Markowich-Toscani-Unterreiter01,MR2609029,MR3247073} in relation with spectral estimates or the \emph{carr\'e du champ} methods. For \emph{carr\'e du champ} techniques in kinetic equations, we can refer to~\cite{MR3677826}, also~\cite{Monmarche2018,MR3438742}, and finally Remark~6.7 in~\cite{MR2381160} for an early contribution on $\varphi$-entropies. Although $\varphi$-entropies are natural in the context of the kinetic Fokker-Planck equation, precise connections were made only quite recently. In~\cite{2014arXiv1409.5425A}, A.~Arnold and J.~Erb discuss $\varphi$-entropies in the context of the kinetic Fokker-Planck equation and prove, among more general results, Proposition~\ref{Prop:AE}. We can also refer to~\cite{MR3557714,MR3468699,Monmarche2018} for various related results. As far as we know, no result such as Theorem~\ref{Thm:Main} has been established yet.

\section{A review of results on \texorpdfstring{$\varphi$}{varphi}-entropies}\label{Sec:PhiEntropy}

In this section we consider a $\varphi$-entropy defined by $\mathcal E[w]:=\idg{\varphi(w)}$ where $d\gamma=e^{-\potential}\,dx$ is a probability measure and $\varphi$ satisfies~\eqref{Ass:Convexity}. Most of the results presented here are known, but they are scattered in the literature. Our purpose here is to collect some essential statements and present simple proofs.

\subsection{Generalized Csisz\'ar-Kullback-Pinsker inequality}\label{sec:CKgen}

By assumption~\eqref{Ass:Convexity}, we know that $\mathcal E$ is nonnegative and achieves its minimum at $w\equiv1$. It results from the strict convexity of $\varphi$ that $\mathcal E[w]$ controls a norm of $(w-1)$ under a generic assumption compatible with the expression of $\varphi_p$. The classical result of~\cite{MR0213190,Csiszar67,Kullback67} has been extended in\cite{kemperman1969optimum,MR1801751,Caceres-Carrillo-Dolbeault02,csiszar2011information}. Here is a statement, with a short proof taken from Section~1.4 of~\cite{BDIK}, for completeness.
\begin{proposition}\label{prop:CK} Let $p\in[1,2]$, $w\in\mathrm L^1\cap\mathrm L^p(\R^d,d\gamma)$ be a nonnegative function, and assume that $\varphi\in C^2(0,+\infty)$ is a nonnegative strictly convex function such that $\varphi(1)=\varphi'(1)=0$. If $A:=\inf_{s\in(0,\infty)}s^{2-p}\,\varphi''(s)>0$, then
\[\label{CKgen}
\mathcal E[w]\ge2^{-\frac2p}\,A\,\min\,\left\{1,\nrmg wp^{p-2}\right\}\,\nrmg{w-1}p^2\,.
\]\end{proposition}
When $\varphi=\varphi_p$, we find that $A=p$. This inequality has many variants and extensions: it is not limited to $\R^d$ but also holds on bounded domains or manifolds and the relative $\varphi$-entropy $\idg{\big(\varphi(w_1)-\varphi(w_2)-\varphi'(w_1)\,(w_2-w_1)\big)}$ can also be used to measure $\nrmg{w_2-w_1}p^2$.

\begin{proof} Up to the addition of a small constant, we can assume that $w>0$ and argue by density. A Taylor expansion at order two shows that
\[
\mathcal E[w]=\frac12\idg{\varphi''(\xi)\,|w-1|^2}\ge\frac A2\idg{\xi^{p-2}\,|w-1|^2}
\]
where $\xi$ lies between $1$ and $w$. With $\alpha=p\,(2-p)/2$ and $h>0$, for any measurable set ${\mathcal A}\subset\R^d$, we get
\[
\int_{\mathcal A}|w-1|^p\,h^{-\alpha}\,h^\alpha\,d\gamma\le\(\int_{\mathcal A}|w-1|^2\,h^{p-2}\,d\gamma\)^{p/2}\(\int_{\mathcal A}h^p\,d\gamma\)^{(2-p)/2}
\]
by H\"older's inequality. We apply this formula to two different sets.\\
On $\mathcal A=\{x\in\R^d\,:\,w(x)>1\}$, we use $\xi^{p-2}\ge w^{p-2}$ and take $h=w$:
\[
\int_{\{w>1\}}|w-1|^2\,\xi^{p-2}\,d\gamma\ge\(\int_{\{w>1\}}|w-1|^p\,d\gamma\)^{2/p}\,\nrmg wp^{p-2}\,.
\]
On $\mathcal A=\{x\in\R^d\,:\,w(x)\le1\}$, we use $\xi^{p-2}\ge1$ and take $h=1$:
\[
\int_{\{w\le1\}}|w-1|^2\,\xi^{p-2}\,d\gamma\ge\(\int_{\{w\le1\}}|w-1|^p\,d\gamma\)^{2/p}\,.
\]
By adding these two estimates and using with $r=2/p\ge1$ the elementary inequality $(a+b)^r\le2^{r-1}(a^r+b^r)$ for any $a$, $b\ge0$ allows us to conclude the proof.\end{proof}

\subsection{Convexity, tensorization and sub-additivity}\label{Sec:Tensorization}

Let us turn our attention to~\eqref{Ineq:E-EP}. To start with, we observe that the functional $w\mapsto\mathcal I[w]=\idg{\varphi''(w)\,|\nabla w|^2}$ is convex if and only if $1/\varphi''$ is concave. Now let us consider two probability measures $d\gamma_1$ and $d\gamma_2$ defined respectively on $\R^{d_1}$ and $\R^{d_2}$, such that Inequality~\eqref{Ineq:E-EP} holds with $\gamma=\gamma_i$, and $i=1$, $2$:
\be{Ineq:E-EP2}
\int_{\R^{d_i}}\varphi''(w)\,|\nabla w|^2\,d\gamma_i=:\mathcal I_{\gamma_i}[w]\ge\Lambda_i\,\mathcal E_{\gamma_i}[w]\quad\forall\,w\in\mathrm H^1(\R^{d_i},d\gamma_i)\,,
\ee
Here we denote by $\mathcal E_\gamma$ the $\varphi$-entropy for functions which are not normalized, that~is,
\[
\mathcal E_\gamma[w]:=\idg{\varphi(w)}-\varphi\(\idg w\)\,.
\]
Assuming that $d\gamma$ is a probability measure, by Jensen's inequality we know that $w\mapsto\mathcal E_\gamma[w]$ is nonnegative because $\varphi$ is convex. As we shall see below, $w\mapsto\mathcal E_\gamma[w]$ is also convex, which is the key ingredient for \emph{tensorization}. The question at stake is to know if Inequality~\eqref{Ineq:E-EP} holds on $\R^{d_1}\times\R^{d_2}$ for the measure $d\gamma=d\gamma_1\otimes\gamma_2$. Most of the results of Section~\ref{Sec:Tensorization} have been stated in~\cite{MR2081075} or are considered as classical. Our contribution here is to give simplified proofs.
\begin{theorem}\label{Thm:TensPhi} Assume that $\varphi$ satisfies~\eqref{Ass:Convexity}. If $d\gamma_1$ and $d\gamma_2$ are two probability measures on $\R^{d_1}\times\R^{d_2}$ satisfying~\eqref{Ineq:E-EP2} with positive constants $\Lambda_1$ and $\Lambda_2$, then $d\gamma_1\otimes\gamma_2$ is such that the following inequality holds:
\begin{multline*}
\mathcal I_{\gamma_1\otimes\gamma_2}[w]=\int_{\R^{d_1}\times\R^{d_2}}\varphi''(w)\,|\nabla w|^2\,d\gamma_1\,d\gamma_2\\
\ge\min\{\Lambda_1,\Lambda_2\}\,\mathcal E_{\gamma_1\otimes\gamma_2}[w]\quad\forall\,w\in\mathrm H^1(\R^{d_1}\times\R^{d_2},d\gamma)\,.
\end{multline*}
\end{theorem}
It is straightforward to notice that the Fisher information is additive
\[
\mathcal I_{\gamma_1\otimes\gamma_2}[w]=\iomutwo{\mathcal I{\gamma_1}[w]}+\iomuone{\mathcal I{\gamma_2}[w]}\,,
\]
so that the proof of Theorem~\ref{Thm:TensPhi} can be reduced to the proof of a \emph{sub-additivity} property of the $\varphi$-entropies that goes as follows.
\begin{proposition}\label{Prop:TensPhi} Assume that $\varphi$ satisfies~\eqref{Ass:Convexity} and consider two probability measures $d\gamma_1$ and $d\gamma_2$ on $\R^{d_1}\times\R^{d_2}$. Then for any $w\in\mathrm L^1(\R^{d_1}\times\R^{d_2},d\gamma_1\otimes\gamma_2)$, we have
\[
\mathcal E_{\gamma_1\otimes\gamma_2}[w]\le\iomutwo{\mathcal E_{\gamma_1}[w]}+\iomuone{\mathcal E_{\gamma_2}[w]}\quad\forall\,w\in\mathrm L^1(d\gamma_1\otimes\gamma_2)\,.
\]
\end{proposition}
This last result relies on convexity properties that we are now going to study. As a preliminary step, we establish an inequality of Jensen type.
\begin{lemma}\label{Lem:Jensen} Let $w\in\mathrm L^1(\R^{d_1}\times\R^{d_2},d\gamma_1\otimes\gamma_2)$ be a function of two variables $(x_1,x_2)\in\R^{d_1}\times\R^{d_2}$. If $\mathcal F_{\gamma_1}$ is a convex functional on $\mathrm L^1(d\gamma_1)$ such that
\be{Eqn:ii}
\frac d{dt}{\int_{\R^{d_2}}\mathcal F_{\gamma_1}\left[t\,w+(1-t)\,\textstyle{\int_{\R^{d_2}}}\,w\,d{\gamma_2}\right]\,d{\gamma_2}}_{|t=0}\kern -5pt=0\,,
\ee
then the following inequality holds:
\[
\int_{\R^{d_2}} \mathcal F_{\gamma_1}[w]\,d\gamma_2\ge\mathcal F_{\gamma_1}\left[\int_{\R^{d_2}}w\,d\gamma_2\right]\,.
\]
\end{lemma}
\begin{proof} Let $w_t=t\,w+(1-t)\iomutwo w$. By convexity of $\mathcal F_{\gamma_1}$,
\[
\mathcal F_{\gamma_1}[w_t]\le t\,\mathcal F_{\gamma_1}[w]+(1-t)\,\mathcal F_{\gamma_1}\left[\iomutwo w\right]\,.
\]
Hence it follows that
\[
\mathcal F_{\gamma_1}[w_t]-\mathcal F_{\gamma_1}\left[\iomutwo w\right]\le t\,\(\mathcal F_{\gamma_1}[w]-\mathcal F_{\gamma_1}\left[\iomutwo w\right]\)\,,
\]
from which we deduce that
\[
0=\frac d{dt}\mathcal F_{\gamma_1}[w_t]_{|t=0}\le\mathcal F_{\gamma_1}[w]-\mathcal F_{\gamma_1}\left[\iomutwo w\right]\,.
\]
Conclusion holds after integrating with respect to $\gamma_2$. \end{proof}

The second observation is the proof of the convexity of $w\mapsto\mathcal E_\gamma[w]$. The following result is taken from~\cite{MR1796718}.
\begin{lemma}\label{Lem:ConvPhi} If $\varphi$ satisfies~\eqref{Ass:Convexity}, then $\mathcal E_\gamma$ is convex. \end{lemma}
\begin{proof} We give a two steps proof of this result, for completeness.
\\
$\bullet$ Define $x_t=t\,y+(1-t)\,x$, $t\in(0,1)$. Since $1/\varphi''$ is concave,
\be{Concavity}
\frac 1{\varphi''(x_t)}\ge\frac t{\varphi''(y)}+\frac{1-t}{\varphi''(x)}\,.
\ee
The function $\varphi$ is convex, hence $\varphi''(x)>0$ and $\varphi''(y)>0$ and so
\[
\frac 1{\varphi''(x_t)}\ge\frac t{\varphi''(y)}\quad\mbox{and}\quad\frac 1{\varphi''(x_t)}\ge\frac{1-t}{\varphi''(x)}\,.
\]
This means
\[
\varphi''(y)\ge t\,\varphi''(x_t)\quad\mbox{and}\quad\varphi''(x)\ge(1-t)\,\varphi''(x_t)\,.
\]
We can also rewrite \eqref{Concavity} as
\[
\varphi''(x)\,\varphi''(y)\ge(t\,\varphi''(x)+(1-t)\,\varphi''(y))\,\varphi''(x_t)\,.
\]
Consider the function
\[
F_t(x,y):=t\,\varphi(y)+(1-t)\,\varphi(x)-\varphi(x_t)
\]
and observe that
\[
\mbox{\rm Hess}(F_t)=\(\begin{array}{cc}(1-t)\,\varphi''(x)-(1-t)^2\,\varphi''(x_t)\quad&-\,t\,(1-t)\,\varphi''(x_t)\cr-\,
t\,(1-t)\,\varphi''(x_t)&t\,\varphi''(y)-t^2\,\varphi''(x_t)\end{array}\)
\]
is nonnegative since both diagonal terms are nonnegative and the determinant is nonnegative. The matrix $\mbox{\rm Hess}(F_t)$ is therefore nonnegative and $F_t$ is convex.
\\
$\bullet$ We observe that
\begin{multline*}
t\,\mathcal E_\gamma[w_1]+(1-t)\,\mathcal E_\gamma[w_0]-\mathcal E_\gamma[t\,w_1+(1-t)\,w_0]\\
=\idg{F_t(w_1,w_0)}-F_t\(\idg{w_1},\idg{w_0}\)
\end{multline*}
is nonnegative by Jensen's inequality, which proves the result.
\end{proof}

\noindent{\bf Proof of Proposition~\ref{Prop:TensPhi}.} We claim that $\mathcal F_{\gamma_1}=\mathcal E_{\gamma_1}$ satisfies~\eqref{Eqn:ii}. Indeed, let us consider $w_t=t\,w+(1-t)\,w_0$ with $w_0:=\iomutwo w$. A simple computation shows that
\[
\frac d{dt} \mathcal F_{\gamma_1}[w_t]=\iomuone{\varphi'(w_t)\,\(w-w_0\)}-\varphi'\(\iomuone{w_t}\)\iomuone{\(w-w_0\)}\,,
\]
and, as a consequence at $t=0$,
\[
\frac d{dt} \mathcal F_{\gamma_1}[w_t]_{|t=0}=\iomuone{\varphi'(w_0)\,\(w-w_0\)}-\varphi'\(\iomuone{w_0}\)\iomuone{\(w-w_0\)}\,.
\]
Since $w_0$ does not depend on $x_2$, an integration with respect to $\gamma_2$ concludes the proof of~\eqref{Eqn:ii}. From Lemma~\ref{Lem:Jensen}, we get
\[
\int_{\R^{d_2}} \mathcal E_{\gamma_1}[w]\,d\gamma_2\ge\mathcal E_{\gamma_1}\left[\int_{\R^{d_2}}w\,d\gamma_2\right]\,.
\]
By definition of $\mathcal E_{\gamma_1}$, this means
\begin{multline*}
\iomutwo{\left[\iomuone{\varphi(w)}-\varphi\(\iomuone w\)\right]}\\
\ge\iomuone{\varphi\(\iomutwo w\)}-\varphi\(\iomuonetwo w\)\,,
\end{multline*}
from which we deduce
\begin{multline*}
\iomutwo{\left[\iomuone{\varphi(w)}-\varphi\(\iomuone w\)\right]}\\+\iomuone{\left[\iomutwo{\varphi(w)}-\varphi\(\iomutwo w\)\right]}\\
\ge\iomuonetwo{\varphi\(w\)}-\varphi\(\iomuonetwo w\)\,.
\end{multline*}
This ends the proof of Proposition~\ref{Prop:TensPhi}. \finprf

\noindent {\bf Proof of Theorem~\ref{Thm:TensPhi}.} The proof is an easy consequence of Proposition~\ref{Prop:TensPhi} and of the observation that
\begin{multline*}
\min\{\Lambda_1,\Lambda_2\}\,\mathcal E_{\gamma_1\otimes\gamma_2}[w]\\
\le\Lambda_1\iomutwo{\mathcal E_{\gamma_1}[w]}+\Lambda_2\iomuone{\mathcal E_{\gamma_2}[w]}\hspace*{2cm}\\\hspace*{2cm}
\le\iomuonetwo{\varphi''(w)\,\big[\,|\nabla_{x_1}w|^2+|\nabla_{x_2}w|^2\,\big]}\\
\le\iomuonetwo{\varphi''(w)\,|\nabla w|^2}=\mathcal I_{\gamma_1\otimes\gamma_2}[w]\,.
\end{multline*}
\finprf

As a concluding remark, we observe that tensorization is not limited to probability measures on $\R^d$. The main interest of such an approach when dealing with $\R^d$ is that it is enough to establish the inequality when $d=1$. In the case $d=1$, sharp criteria can be found in~\cite{MR2052235} (also see~\cite{MR2320410}). There are many related issues that can be traced back to the work of Muckenhoupt, \emph{e.g.},~\cite{MR0311856} and Hardy (see~\cite{MR944909}).

\subsection{Entropy -- entropy production inequalities: perturbation results}\label{Sec:HS}

Perturbing the measure in the case of a Poincar\'e inequality is essentially trivial. In the case of the logarithmic Sobolev inequality, this has been done by Holley and Stroock in~\cite{MR893137}. More general entropy functionals have been considered in~\cite{MR1801751}, which cover all $\varphi$-entropies. Also see~\cite{Gentil00,MR2081075}.

Assume that for some probability measure $d\gamma$ and for some $\Lambda>0$, Inequality~\eqref{Ineq:E-EP} holds, that is,
\be{IneqConvLogSob}
\Lambda\left[\idg{\varphi(w)}-\varphi(\overline w)\right]\le\idg{\varphi''(w)|\nabla w|^2}\quad\forall\,w\in\mathrm H^1(d\gamma)\,.
\ee
Here we denote by $\overline w$ the average of $w$ with respect to $d\gamma$: $\overline w:=\idg w$. Assume that $d\mu$ is a measure which is absolutely continuous with respect to $d\gamma$ and such that
\[
e^{-b}\,d\gamma\le d\mu\le e^{-a}\,d\gamma
\]
for some constants $a$, $b\in\R$. The statement below generalizes the one of Lemma~5.2 of~\cite{blanchet:46}.
\begin{lemma}\label{Lem:Holley-Stroock} Under the above assumption, if $\varphi$ is a $C^2$ function such that $\varphi''>0$, then
\[
e^{a-b}\,\Lambda\int_{\R^d}\big[\varphi(w)-\varphi(\widetilde w)-\varphi'(\widetilde w)(w-\widetilde w)\big]\,d\mu\le\int_{\R^d}\varphi''(w)\,|\nabla w|^2\,d\mu\quad\forall\,w\in\mathrm H^1(d\mu)\,,
\]
where $\widetilde w:=\int_{\R^d}w\,d\mu\,/\int_{\R^d}d\mu$.\end{lemma}
\begin{proof} We start by observing that
\begin{multline*}
e^b\int_{\R^d}\varphi''(w)|\nabla w|^2\,d\mu\ge\idg{\varphi''(w)|\nabla w|^2}=\mathcal I_\gamma[w]\\
\hspace*{1cm}\ge\Lambda\,\mathcal E_\gamma[w]=\Lambda\left[\idg{\varphi(w)}-\varphi(\overline w)\right]\\
=\Lambda\idg{\(\varphi(w)-\varphi(\overline w)-\varphi'(\overline w)\,(w-\overline w)\)}\,.
\end{multline*}
By convexity of $\varphi$, we know that $\varphi(w)-\varphi(\overline w)-\varphi'(\overline w)\,(w-\overline w)\ge0$, so that
\begin{multline*}
\Lambda\,\mathcal E_\gamma[w]\ge\Lambda\,e^a\int_{\R^d}\(\varphi(w)-\varphi(\overline w)-\varphi'(\overline w)\,(w-\overline w)\)d\mu\\
=\Lambda\,e^a\int_{\R^d}\(\varphi(w)-\varphi(\overline w)-\varphi'(\overline w)\,(\widetilde w-\overline w)\)d\mu\,.
\end{multline*}
By convexity of $\varphi$ again, $\varphi(\overline w)+\varphi'(\overline w)\,(\widetilde w-\overline w)\le\varphi(\widetilde w)$, which shows that
\[
\Lambda\,\mathcal E_\gamma[w]\ge\Lambda\,e^a\int_{\R^d}\(\varphi(w)-\varphi(\widetilde w)\)d\mu=e^a\,\Lambda\int_{\R^d}\big[\varphi(w)-\varphi(\widetilde w)-\varphi'(\widetilde w)(w-\widetilde w)\big]\,d\mu
\]
and completes the proof. \end{proof}

\subsection{Entropy -- entropy production inequalities and linear flows}\label{Sec:EEP}

Let us consider the counterpart of the \emph{Ornstein-Uhlenbeck equation}~\eqref{Eqn:OU} on a smooth convex bounded domain $\Omega$
\be{Eqn:OUbded}
\frac{\partial w}{\partial t}=\mathsf L\,w:=\Delta w-\nabla\potential\cdot\nabla w\,,
\ee
supplemented with homogenous Neumann boundary conditions
\[
\nabla w\cdot\nu=0\quad\mbox{on}\quad\partial\Omega\,,
\]
where $\nu$ denotes a unit outward pointing normal vector orthogonal to $\partial\Omega$. Let us consider the measure $d\gamma=\(\int_\Omega e^{-\potential}\,dx\)^{-1}e^{-\potential}\,dx$. If $w$ solves~\eqref{Eqn:OUbded} with a nonnegative initial datum $w_0$ such that $\idgb{w_0}=1$, then mass is conserved so that $\idgb{w(t,\cdot)}=1$ for any $t\ge0$ and converges to $1$ as $t\to+\infty$. The next question is how to measure the rate of convergence using the $\varphi$-entropy. For simplicity, let us assume that $\varphi=\varphi_p$ for some $p\in[1,2]$. An answer is given by the formal computation of Section~\ref{Sec:Intro}, adapted to the bounded domain $\Omega$. Because of the boundary condition, it is straightforward to check that
\[
\frac d{dt}\idgb{\frac{w^p-1}{p-1}}=-\frac4p\idgb{|\nabla w^{p/2}|^2}
\]
if $p>1$ and a similar results holds when $p=1$. Hence, if for some $\Lambda>0$ we can prove that
\be{Ineq:E-EPb}
\idgb{\frac{w^p-1}{p-1}}\le\frac4{p\,\Lambda}\idgb{|\nabla w^{p/2}|^2}\quad\mbox{for any $w$ such that}\quad\idgb w=1\,,
\ee
then we can conclude that $\idgb{\frac{w^p-1}{p-1}}$ decays like $e^{-\Lambda\,t}$. The main idea of the Bakry-Emery method, or \emph{carr\'e du champ} method, as it is exposed in~\cite{Bakry-Emery85} is that~\eqref{Ineq:E-EPb} can be established using the flow itself, by computing $\frac d{dt}\idgb{|\nabla z|^2}$ with $z:=w^{p/2}$. Let us sketch the main steps of the proof.

As a preliminary observation, we notice that $\mathsf L$ is self-adjoint in $\mathrm L^2(\Omega,d\gamma)$ in the sense that
\[
\idgb{w_1\,(\mathsf L\,w_2)}=-\idgb{\nabla w_1\cdot\nabla w_2}=\idgb{(\mathsf L\,w_1)\,w_2}
\]
and also that
\[
[\nabla,\,\mathsf L]=-\,\mathrm{Hess}\,\potential\,.
\]
Using $w=z^{2/p}$ we deduce from~\eqref{Eqn:OUbded} that
\be{Eqn:z}
\frac{\partial z}{\partial t}=\mathsf L\,z+\frac{2-p}p\,\frac{|\nabla z|^2}z\,.
\ee
We adopt the convention that $a\cdot b=\sum_{i=1}^da_i\,b_i$ if $a=(a_i)_{1\le i\le d}$ and $b=(b_i)_{1\le i\le d}$ are two vectors with values in $\R^d$. If $m=(m_{i,j})_{1\le i,j\le d}$ and $n=(n_{i,j})_{1\le i,j\le d}$ are two matrices, then $m:n=\sum_{i,j=1}^dm_{i,j}\,n_{i,j}$. Also $a\otimes b$ denotes the matrix $(a_i\,b_j)_{1\le i,j\le d}$. We shall use $|a|^2=a\cdot a$ and $\|m\|^2=m:m$ for vectors and matrices respectively. With these notations, let us use~\eqref{Eqn:z} to compute
\begin{eqnarray*}
&&\hspace*{-8pt}\frac12\,\frac d{dt}\idgb{|\nabla z|^2}=\idgb{\nabla z\cdot\nabla\(\mathsf L\,z+\frac{2-p}p\,\frac{|\nabla z|^2}z\)}\\
&=&\idgb{\nabla z\cdot\(\mathsf L\,\nabla z-\,\mathrm{Hess}\,\potential\,\nabla z\)}+\,\frac{2-p}p\idgb{\nabla z\cdot\(2\,\mathrm{Hess}\,z\,\frac{\nabla z}z-\frac{|\nabla z|^2}z\,\nabla z\)}\\
&=&-\idgb{\big\|\mathrm{Hess}\,z\big\|^2}-\idgb{\mathrm{Hess}\,\potential:\nabla z\otimes\nabla z}+\int_{\partial\Omega}\mathrm{Hess}\,z:\nabla z\otimes\nu\,e^{-\potential}\,d\sigma\\\\
&&+\,2\,\frac{2-p}p\idgb{\mathrm{Hess}\,z:\frac{\nabla z\otimes\nabla z}z}-\,\frac{2-p}p\idgb{\left\|\frac{\nabla z\otimes\nabla z}z\right\|^2}\\
&=&-\frac2p\,(p-1)\idgb{\big\|\mathrm{Hess}\,z\big\|^2}-\idgb{\mathrm{Hess}\,\potential:\nabla z\otimes\nabla z}\\
&&-\,\frac{2-p}p\idgb{\left\|\mathrm{Hess}\,z-\frac{\nabla z\otimes\nabla z}z\right\|^2}+\int_{\partial\Omega}\mathrm{Hess}\,z:\nabla z\otimes\nu\,e^{-\potential}\,d\sigma\,.
\end{eqnarray*}
Here $d\sigma$ denotes the surface measure induced by Lebesgue's measure on $\partial\Omega$. We learn from Grisvard's lemma, see for instance Lemma~5.1 in~\cite{GTS} or~\cite{MR775683}, that $\int_{\partial\Omega}\mathrm{Hess}\,z:\nabla z\otimes\nu\,e^{-\potential}\,d\sigma$ is nonpositive as soon as $\Omega$ is convex and $\nabla z\cdot\nu=0$ on $\partial\Omega$. As soon as we know that either
\[
\mathrm{Hess}\,\potential\ge\Lambda_\star\,\mathrm{Id}
\]
for some $\Lambda_\star>0$, or the inequality
\[
\frac2p\,(p-1)\idgb{|\nabla X|^2}+\idgb{\mathrm{Hess}\,\potential:X\otimes X}\ge\Lambda(p)\idgb{|X|^2}\quad\forall\,X\in\mathrm H^1(\Omega,d\gamma)^d
\]
holds for some $\Lambda(p)>0$, which is a weaker assumption for any $p>1$, then we obtain that
\[
\frac d{dt}\idgb{|\nabla z|^2}\le-\,2\,\Lambda(p)\idgb{|\nabla z|^2}\,.
\]
Of course we know that $\Lambda(p)\ge\Lambda_\star$. By convention, we take $\Lambda(1)=\Lambda_\star$.
\begin{proposition}\label{Prop:BE} Assume that $p\in[1,2]$, $\varphi=\varphi_p$ and, with the above notations, $\Lambda(p)>0$. If $\Omega$ is a smooth convex bounded domain in $\R^d$, then~\eqref{Ineq:E-EPb} holds with $\Lambda=2\,\Lambda(p)$.\end{proposition}
\begin{proof} It is straightforward. In view of the above computations, we know that
\[
\frac d{dt}\(\frac4{p\,\Lambda}\idgb{|\nabla w^{p/2}|^2}-\idgb{\frac{w^p-1}{p-1}}\)\le0
\]
and $\lim_{t\to+\infty}\idgb{\frac{w^p-1}{p-1}}=\lim_{t\to+\infty}\idgb{|\nabla w^{p/2}|^2}=0$. This is enough to conclude that, for any $t\ge0$,
\[
\frac4{p\,\Lambda}\idgb{|\nabla w^{p/2}|^2}-\idgb{\frac{w^p-1}{p-1}}\ge0\,.
\]
\end{proof}

We conclude this section with the unbounded case $\Omega=\R^d$. For any given $p\in[1,2]$, let us assume that the inequality
\[
\frac2p\,(p-1)\kern -3pt\idg{|\nabla X|^2}+\idg{\mathrm{Hess}\,\potential:X\otimes X}\ge\Lambda(p)\kern -3pt\idg{|X|^2}\quad\forall\,X\in\mathrm H^1(\R^d,d\gamma)^d
\]
holds for some $\Lambda(p)>0$. For $p>1$, this assumption is a spectral gap condition on a vector valued Schr\"odinger operator: see for instance~\cite{DNS} for further details. With this assumption in hand, we have the following functional inequality, which interpolates between the logarithmic Sobolev inequality and the Poincar\'e inequality.
\begin{corollary}\label{Cor:BE} Assume that $q\in[1,2)$ and let us consider the probability measure $d\gamma=e^{-\potential}\,dx$ on $\R^d$. Then with $\Lambda=\Lambda(2/q)$, we have
\be{Ineq:InterpGaussian}
\frac{\nrmg f2^2-\nrmg fq^2}{2-q}\le\frac1\Lambda\idg{|\nabla f|^2}\quad\forall\,f\in\mathrm H^1(\R^d,d\gamma)\,.
\ee\end{corollary}
\begin{proof} By homogeneity, we know from Proposition~\ref{Prop:BE} that
\[
\idgb{\frac{w^p-\overline w^p}{p-1}}\le\frac2{p\,\Lambda(p)}\idgb{|\nabla w^{p/2}|^2}
\]
for all $w$ such that $f=w^{p/2}$. Here we take $p=2/q$. The conclusion holds by approximating $\R^d$ by a growing sequence of bounded convex domains.\end{proof}

An equivalent form of~\eqref{Ineq:InterpGaussian} is
\be{Ineq:InterpGaussian2}
\mathcal I[w]\ge\Lambda\,\mathcal E[w]\quad\forall\,w\in\mathrm H^1(\R^d,d\gamma)\;\mbox{such that}\;\idg w=1
\ee
with the notation of Section~\ref{Sec:Intro}, $\varphi=\varphi_p$ and $p=2/q\in[1,2]$.

\begin{remark}\label{Rem:OptGaussian} The optimality of the constant $\Lambda=1$ in~\eqref{Ineq:InterpGaussian} is easy to obtain when $\potential(x)=\frac12\,|x|^2$. With $q=1$,~\eqref{Ineq:InterpGaussian} is the Gaussian Poincar\'e inequality
\[
\nrmg{f-\bar f}2^2\le\idg{|\nabla f|^2}\quad\forall\,f\in\mathrm H^1(\R^d,d\gamma)\quad\mbox{with}\quad\bar f=\idg f\,,
\]
with equality if $f=f_1$, $f_1(x)=x_1$. By taking the limit as $q\to2_-$ in~\eqref{Ineq:InterpGaussian}, we recover Gross' logarithmic Sobolev inequality
\[
\idg{f^2\,\log\(\frac{f^2}{\nrmg f2^2}\)}\le2\idg{|\nabla f|^2}\quad\forall\,f\in\mathrm H^1(\R^d,d\gamma)\,.
\]
For any $q\in[1,2)$, the equality case in~\eqref{Ineq:InterpGaussian} with $\Lambda=1$ is achieved by considering $1+\varepsilon\,f_1$ as a test function in the limit as $\varepsilon\to0$.

From the point of view of the evolution equation, it is easy to see that the equality in~\eqref{Ineq:E-EPb} is achieved asymptotically as $t\to+\infty$ by taking $w=u/u_\star$ where~$u$ is the solution of~\eqref{Eqn:FP} given by
\[
u(t,x)=u_\star\(x-x_\star(t)\)
\]
with $x_\star(t)=x_0\,e^{-t}$ for any fixed $x_0\in\R^d$.
\end{remark}

\subsection{Improved entropy -- entropy production inequalities}\label{Sec:iEEP}

In the proof of Proposition~\ref{Prop:BE}, the term $\idg{\left\|\mathrm{Hess}\,z-\nabla z\otimes\nabla z/z\right\|^2}$ has been dropped. In some cases, one can recombine the other terms differently and obtain an improved inequality if $q\in(1,2)$. See~\cite{MR2152502} (and also~\cite{ABD} for a spectral point of view or~\cite{DEKL} in the case of the sphere). The boundary term $\int_{\partial\Omega}\mathrm{Hess}\,z:\nabla z\otimes\nu\,e^{-\potential}\,d\sigma$ may also be of importance, as it is suggested in nonlinear problems by~\cite{1407}.

Let us give an example of an improvement, based on~\cite{MR2152502}, in the special case $\potential(x)=|x|^2/2$. Using $\mathrm{Hess}\,\potential=\mathrm{Id}$, after approximating $\R^d$ by bounded domains, we obtain that
\begin{multline*}
\frac12\,\frac d{dt}\idg{|\nabla z|^2}+\idg{|\nabla z|^2}\le-\idg{\left\|\mathrm{Hess}\,z-\,\frac{2-p}p\,\frac{\nabla z\otimes\nabla z}z\right\|^2}\\
-\,\frac2p\,\kappa_p\idg{\frac{|\nabla z|^4}{z^2}}
\end{multline*}
with $\kappa_p=(p-1)\,(2-p)/p$. A simple Cauchy-Schwarz inequality shows that
\[
\(\idg{|\nabla z|^2}\)^2\le\idg{\frac{|\nabla z|^4}{z^2}}\idg{z^2}\,.
\]
With the notations of Section~\ref{Sec:Intro}, we have $\idg{z^2}=\idg{w^p}=1+(p-1)\,\mathcal E[w]$ and $\idg{|\nabla z|^2}=\frac p4\,\mathcal I[w]$ so that
\[
\frac12\,\frac d{dt}\idg{|\nabla z|^2}+\idg{|\nabla z|^2}\le-\,\frac2p\,\kappa_p\,\frac{\(\idg{|\nabla z|^2}\)^2}{\idg{|z|^2}}
\]
can be rewritten as
\be{ImprovedEDO}
\frac d{dt}\mathcal I[w]+2\,\mathcal I[w]\le-\,\kappa_p\,\frac{\mathcal I[w]^2}{1+(p-1)\,\mathcal E[w]}\,.
\ee
We recall that we consider here the case $\varphi=\varphi_p$, $p\in(1,2)$, so that $\kappa_p$ is positive and we can take advantage of~\eqref{ImprovedEDO} to obtain an improved version of Corollary~\ref{Cor:BE}. The following result follows the scheme of Theorem~2 in~\cite{MR2152502}.
\begin{proposition}\label{Prop:AD} Assume that $q\in(1,2)$ and let us consider the Gaussian probability measure $d\gamma=(2\pi)^{-d/2}\,e^{-|x|^2/2}\,dx$. Then there exists a strictly convex function~$F$ on $\R^+$ such that $F(0)=0$ and $F'(0)=1$, for which
\[
\frac1q\,F\(q\,\frac{\nrmg f2^2-1}{2-q}\)\le\nrmg{\nabla f}2^2
\]
for any $f\in\mathrm H^1(\R^d,d\gamma)$ such that $\nrmg fq=1$.\end{proposition}
\begin{proof} The proof follows the strategy of~\cite{MR2152502}. Let $\mathsf e(t):=\frac1{p-1}\(\idg{z^2}-1\)$ where $z=w^{p/2}$ solves~\eqref{Eqn:z} with initial datum $f$. We deduce from~\eqref{ImprovedEDO} that
\[
\mathsf e''+2\,\mathsf e'\ge\frac{\kappa_p\,|\mathsf e'|^2}{1+(p-1)\,\mathsf e}\ge\frac{\kappa_p\,|\mathsf e'|^2}{1+\mathsf e}\,.
\]
The function $F(s):=\frac1{1-\kappa_p}\,\big[1+s-(1+s)^{\kappa_p}\big]$ solves $F'=1+\kappa_p\,\frac F{1+s}$ and we can check that~\eqref{ImprovedEDO} is equivalent to
\[
\frac d{dt}\(\frac{\mathsf e'+2\,F\big(\mathsf e\big)}{\big(1+\mathsf e\big)^{\kappa_p}}\)\ge0\,.
\]
Since $\lim_{t\to+\infty}\(\mathsf e'(t)+2\,F\big(\mathsf e(t)\big)\)=0$, we have shown that $\mathsf e'+2\,F\big(\mathsf e\big)\le0$ for any $t\ge0$. This is true in particular at $t=0$, with $z(t=0,\cdot)=f$.\end{proof}

{}From the point of view of entropy -- production of entropy inequalities, we have obtained that
\[
\mathcal I[w]\ge2\,F\(\mathcal E[w]\)
\]
where $F$ is a strictly convex function such that $F(0)=0$ and $F'(0)=1$. Using the homogeneity and substituting $f/\nrmg fq$ to $f$, similar estimates have been used in~\cite{MR2152502} to prove that
\[\label{Ineq:InterpGaussian3}
\tfrac2{(2-q)^2}\left[\nrmg f2^2-\nrmg fq^{2(2-q)}\,\nrmg f2^{2(q-1)}\right]\le\nrmg{\nabla f}2^2\;\forall\,f\in\mathrm H^1(\R^d,d\gamma)\,.
\]

\subsection{Interpolation inequalities: comments and extensions}\label{Sec:IP}

The inequality of Corollary~\ref{Cor:BE} appears in many papers. It is proved for the first time by the \emph{carr\'e du champ} method and any $q\in[1,2]$ in~\cite{Bakry-Emery85} in the case of a compact manifold, but special cases were known long before. For instance the case $q=2$ corresponding to the logarithmic Sobolev inequality can be traced back to~\cite{Gross75,Federbush} (also see~\cite{MR479373,MR3493423} for related issues) but was already known as the Blachmann-Stam inequality~\cite{MR0109101}: see~\cite{Villani2008,MR3255069} for a more detailed historical account. The case $q=1$ when $\potential(x)=\frac12\,|x|^2$ is known as the Gaussian Poincar\'e inequality. It appears for instance in~\cite{Nash58} but was probably known much earlier in the framework of the theory of Hermite functions. In the case $q\in(1,2)$ when $\potential(x)=\frac12\,|x|^2$, we may refer to~\cite{MR954373} for a proof based on spectral methods, which has been extended in~\cite{ABD} to more general potentials.

One of the technical limitations of the \emph{carr\'e du champ} method is the difficulty of controlling the boundary terms in the various integrations by parts. In the above proof, we used Grisvard's lemma for convex domains. Alternative methods, which will not be exposed here, rely on the properties of Green's functions, or use direct spectral estimates.

\medskip Let us list some possible extensions:\\
$\bullet$ In Corollary~\ref{Cor:BE}, for any given $q\in[1,2]$, we need that $\Lambda(p)$ is positive only for $p=2/q$. The condition for $p=1$, which is equivalent to $\mathrm{Hess}\,\potential\ge\Lambda(1)\,\mathrm{Id}$ with $\Lambda(1)>0$, is not required unless $q=2$. For any $q<2$, the positivity condition of $\Lambda(2/q)$ is a nonlocal condition, which allows $\potential$ to be a non-uniformly strictly convex potential: see~\cite{DNS} for details.\\
$\bullet$ The case of unbounded convex domains can be considered. Reciprocally, according to~\cite{Arnold-Markowich-Toscani-Unterreiter01}, the case of a bounded convex domain $\Omega$ can be deduced from the Euclidean case, by approximating a function $\potential$ which takes the value $+\infty$ on $\Omega^c$ by smooth locally bounded potentials.\\
$\bullet$ Spectral methods can be used to establish that the family of inequalities of Corollary~\ref{Cor:BE} interpolates between the logarithmic Sobolev inequality and the Poincar\'e inequality: this approach has been made precise in~\cite{MR954373,MR1796718}, with extensions in~\cite{Bartier-Dolbeault,ABD}.\\
$\bullet$ Exhibiting a whole family of Lyapunov functionals for the same evolution equation needs an explanation that has been given in~\cite{DNS2,MR3023408}: to each entropy, we associate a notion of distance such that the equation appears as the gradient flow of the entropy.

\medskip In the context of linear diffusions and Markov processes, $\varphi$-entropies are very natural objects which put the Gibbs entropy and the quadratic form associated to the Poincar\'e inequality in a common framework. It is therefore evident to ask the same question in a kinetic framework involving a degenerate diffusion operator coupled to a transport operator. Much less has been done so far and the next section is a contribution to the issue of optimal rates of convergence measured by $\varphi_p$-entropies, with a special emphasis on $p\neq1$, $2$.

\section{Sharp rates for the kinetic Fokker-Planck equation}\label{Sec:kFPsharp}

In this section, our goal is to provide a computation of the sharp exponential rate in Proposition~\ref{Prop:AE} and establish the improvement of Theorem~\ref{Thm:Main} by generalizing the estimate of Proposition~\ref{Prop:AD} to the kinetic setting. The method follows the strategy of Section~3 of~\cite{MR2275692} in case $p=2$, which is sometimes referred to as the \emph{$\mathrm H^1$ hypocoercivity method of C.~Villani}. This method is also known to cover the case $p=1$. We extend it to any $p\in[1,2]$ and compute the precise algebraic expressions, which allows us to identify the sharp rate. Similar computations have been done in~\cite{2014arXiv1409.5425A,MR3468699,MR3557714,2017arXiv170204168E,Monmarche2018}. According to~\cite{MR2042214} (see earlier references therein), the Green function associated with~\eqref{VFP1} is a Gaussian kernel, so that integrations by parts can be performed on $\R^d\times\R^d$ without any special precaution. 

\subsection{\texorpdfstring{$\mathrm H^1$}{H1} hypocoercive estimates}

Using the notation of Section~\ref{Sec:Intro}, our strategy is to consider the solution $h=g^{p/2}$ of~\eqref{VFP3}, where $g=f/f_\star$, define
\[
\mathcal J[h]:=\iri{|\nabla_vh|^2}+2\,\lambda\iri{\nabla_vh\cdot\nabla_xh}+\munu\iri{|\nabla_xh|^2}
\]
and adjust the parameters $\lambda$ and $\munu$ in order to maximize $\lambda_\star=\lambda_\star(\lambda,\munu)>0$ so that
\[
\frac d{dt}\mathcal J[h(t,\cdot,\cdot)]\le-\,\lambda_\star(\lambda,\munu)\,\mathcal J[h(t,\cdot,\cdot)]\,.
\]
Since~\eqref{VFP2} is linear and preserves positivity, we recall that we can assume that $g$ is nonnegative and such that $\nnrmu g1=1$. Let us define the notations:
\begin{eqnarray*}
&&\H{vv}=\(\frac{\partial^2h}{\partial v_i\,\partial v_j}\)_{1\le i,j\le d}\,,\quad\H{xv}=\(\frac{\partial^2h}{\partial x_i\,\partial v_j}\)_{1\le i,j\le d}\,,\\
&&\M{vv}=\(\frac{\partial\sqrt h}{\partial v_i}\,\frac{\partial\sqrt h}{\partial v_j}\)_{1\le i,j\le d}\,,\quad\M{xv}=\(\frac{\partial\sqrt h}{\partial x_i}\,\frac{\partial\sqrt h}{\partial v_j}\)_{1\le i,j\le d}\,.
\end{eqnarray*}
We start by observing that, up to a few integrations by parts, we obtain the identities
\begin{multline}\label{Est1}
\tfrac12\,\frac d{dt}\iri{|\nabla_vh|^2}\\
=-\iri{\nabla_vh\cdot\nabla_v(v\cdot\nabla_xh-x\cdot\nabla_vh)}+\iri{\nabla_vh\cdot\nabla_v(\Delta_vh-v\cdot\nabla_vh)}\\
+\big(\tfrac2p-1\big)\iri{\nabla_vh\cdot\nabla_v\(\frac{|\nabla_vh|^2}h\)}\\
=-\iri{\nabla_vh\cdot\nabla_xh}-\(\iri{\|\H{vv}\|^2}+\iri{|\nabla_vh|^2}\)\\
+\kappa\iri{\(\H{vv}:\M{vv}-2\,\|\M{vv}\|^2\)}
\end{multline}
with $\kappa=8\,(2-p)/p$,
\begin{multline}\label{Est2}
\tfrac12\,\frac d{dt}\iri{|\nabla_xh|^2}\\
=-\iri{\nabla_xh\cdot\nabla_x(v\cdot\nabla_xh-x\cdot\nabla_vh)}+\iri{\nabla_xh\cdot\nabla_x(\Delta_vh-v\cdot\nabla_vh)}\\
+\big(\tfrac2p-1\big)\iri{\nabla_xh\cdot\nabla_x\(\frac{|\nabla_vh|^2}h\)}\\
=\iri{\nabla_vh\cdot\nabla_xh}-\iri{\|\H{xv}\|^2}+\kappa\iri{\(\H{xv}:\M{xv}-2\,\|\M{xv}\|^2\)}\,,
\end{multline}
and
\begin{multline}\label{Est3}
\frac d{dt}\iri{\nabla_vh\cdot\nabla_xh}=\iri{|\nabla_vh|^2}-\iri{|\nabla_xh|^2}-\iri{\nabla_vh\cdot\nabla_xh}\\
-2\iri{\H{vv}:\H{xv}}\\
+\kappa\iri{\(\H{vv}:\M{xv}+\H{xv}:\M{vv}-4\,\M{vv}:\M{xv}\)}\,.
\end{multline}
Collecting these estimates shows that
\begin{multline*}
-\tfrac12\,\frac d{dt}\mathcal J[h(t,\cdot,\cdot)]\\
=-\tfrac12\,\frac d{dt}\(\iri{|\nabla_vh|^2}+2\,\lambda\iri{\nabla_vh\cdot\nabla_xh}+\munu\iri{|\nabla_xh|^2}\)\\
=(1-\lambda)\iri{|\nabla_vh|^2}+\(1+\lambda-\munu\)\iri{\nabla_vh\cdot\nabla_xh}+\lambda\iri{|\nabla_xh|^2}\\
+\iri{\|\H{vv}\|^2}-\kappa\iri{\(\H{vv}:\M{vv}-2\,\|\M{vv}\|^2\)}\\
+\,2\,\lambda\iri{\H{vv}:\H{xv}}-\kappa\,\lambda\iri{\(\H{vv}:\M{xv}+\H{xv}:\M{vv}-4\,\M{vv}:\M{xv}\)}\\
+\munu\iri{\|\H{xv}\|^2}-\kappa\,\munu\iri{\(\H{xv}:\M{xv}-2\,\|\M{xv}\|^2\)}
\end{multline*}
where $\kappa=8\,(2-p)/p$. This can be rewritten as
\[
-\tfrac12\,\frac d{dt}\iri{X^\perp\cdot\mathfrak M_0\,X}=\iri{X^\perp\cdot\mathfrak M_1\,X}+\iri{Y^\perp\cdot\mathfrak M_2\,Y}
\]
where
\[
\mathfrak M_0=\(
\begin{array}{cc}
1&\lambda\\
\lambda&\munu\\
\end{array}
\)\otimes\,\mathrm{Id}_{\R^d}\,,\quad
\mathfrak M_1=\(
\begin{array}{cc}
1-\lambda&\frac{1+\lambda-\munu}2\\
\frac{1+\lambda-\munu}2&\lambda\\
\end{array}
\)\otimes\,\mathrm{Id}_{\R^d}
\]
and
\[
\mathfrak M_2=\(
\begin{array}{cccc}
1&\lambda&-\frac\kappa2&-\frac{\kappa\,\lambda}2\\
\lambda&\munu&-\frac{\kappa\,\lambda}2&-\frac{\kappa\,\munu}2\\
-\frac\kappa2&-\frac{\kappa\,\lambda}2&2\,\kappa&2\,\kappa\,\lambda\\
-\frac{\kappa\,\lambda}2&-\frac{\kappa\,\munu}2&2\,\kappa\,\lambda&2\,\kappa
\munu\\
\end{array}
\)\otimes\,\mathrm{Id}_{\R^d\times\R^d}
\]
are bloc-matrix valued functions of $(\lambda,\munu)$, and
\[
X=\(\nabla_vh,\nabla_xh\)\,,\quad Y=\(\H{vv},\H{xv},\M{vv},\M{xv}\)\,.
\]
The problem is reduced to a problem of linear algebra, namely to maximize
\[
\lambda_\star(\lambda,\munu):=\min_{X\in\R^{2d}}\frac{X^\perp\cdot\mathfrak M_1(\lambda,\munu)\,X}{X^\perp\cdot\mathfrak M_0(\lambda,\munu)\,X}
\]
on the set of parameters $(\lambda,\munu)\in\R^2$ such that
\[
\min_{Y\in\R^{2d}\times\R^{2d}}\frac{Y^\perp\cdot\mathfrak M_2\,Y}{\|Y\|^2}\ge0\,.
\]Here $X$ and $Y$ now arbitrary vectors and matrices respectively in $\R^{2d}$ and $\R^{2d}\times\R^{2d}$. Elementary computations show that $\lambda$ and $\munu$ must satisfy the condition $\lambda^2\le\munu$ and also that $\lambda_\star(\lambda,\munu)$ achieves its maximum at $(\lambda,\munu)=(\frac12,1)$, so that $\lambda_\star(\frac12,1)=\frac12$. For $(\lambda,\munu)=(\frac12,1)$, $\mathfrak M_1(\frac12,1)=\frac12\,\mathfrak M_0(\frac12,1)$ and the eigenvalues of $\mathfrak M_2(\frac12,1)$ are given as a function of $\kappa=8\,(2-p)/p$ by
\[
\lambda_1(\kappa):=\frac14\(2\,\kappa+1-\,\sqrt{5\,\kappa^2-4\,\kappa+1}\)\,,\;\lambda_2(\kappa):=\frac34\(2\,\kappa+1-\,\sqrt{5\,\kappa^2-4\,\kappa+1}\)\,,
\]
\[
\lambda_3(\kappa):=\frac14\(2\,\kappa+1+\sqrt{5\,\kappa^2-4\,\kappa+1}\)\,,\quad\lambda_4(\kappa):=\frac34\(2\,\kappa+1+\sqrt{5\,\kappa^2-4\,\kappa+1}\)\,.
\]
In the range $p\in[1,2]$, which means $\kappa\in[0,8]$, they are all nonnegative: see Fig.~\ref{Fig:F1}. Since $\lambda_1(\kappa)$ is the lowest eigenvalue, we have proved the following result.
\begin{figure}[ht]
\begin{center}
\includegraphics[width=10cm]{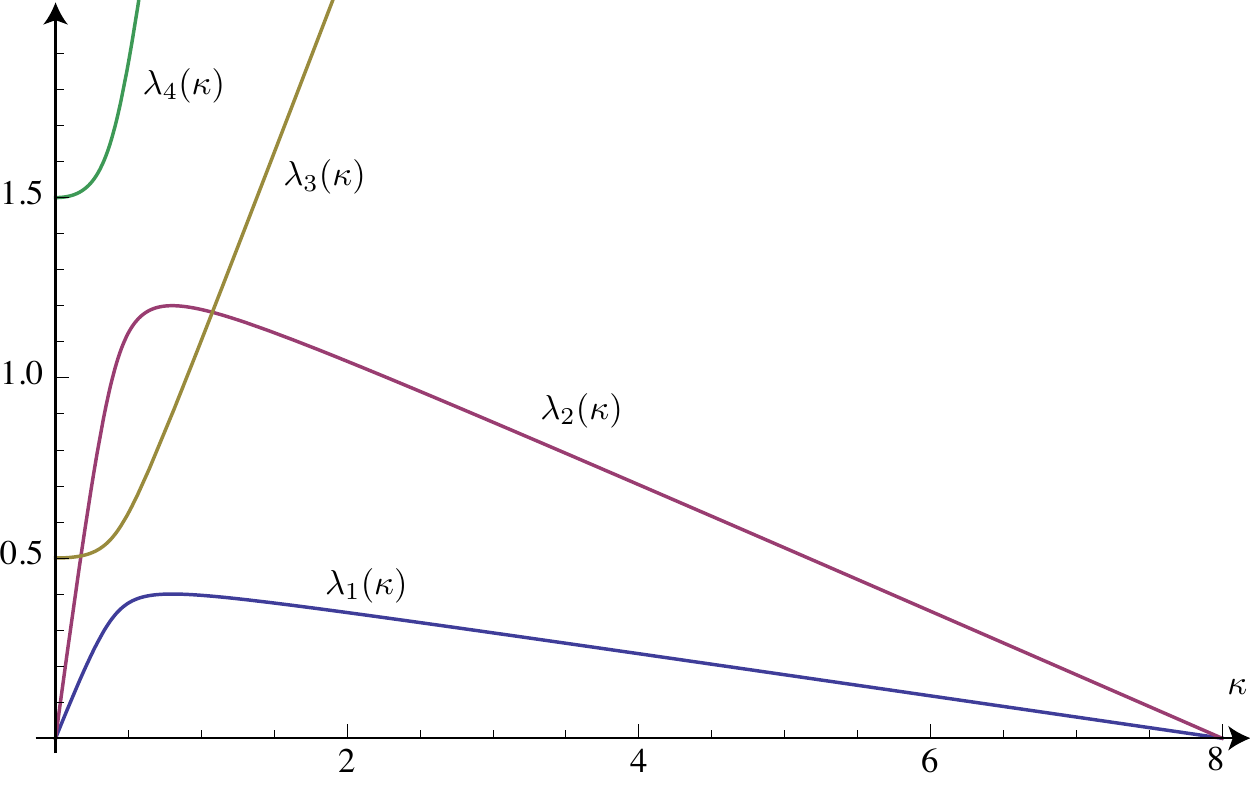}
\caption{\label{Fig:F1} Plot of the eigenvalues of $\mathfrak M_2(\frac12,1)$ as a function of $\kappa$.}
\end{center}
\end{figure}
\begin{lemma}\label{Lem:VFPalg} With the above notations and $(\lambda,\munu)=(\frac12,1)$, we have the estimate
\begin{multline*}
\iri{X^\perp\cdot\mathfrak M_1\,X}+\iri{Y^\perp\cdot\mathfrak M_2\,Y}\ge\frac12\iri{X^\perp\cdot\mathfrak M_0\,X}\hspace*{1cm}\\
\hspace*{1cm}+\frac14\(2\,\kappa+1-\,\sqrt{5\,\kappa^2-4\,\kappa+1}\)|Y|^2\,.
\end{multline*}
\end{lemma}

\subsection{Proof of Proposition \texorpdfstring{\ref{Prop:AE}}{1.1}}

Assume that $h$ solves~\eqref{VFP3}. With $(\lambda,\munu)=(\frac12,1)$, we deduce from Lemma~\ref{Lem:VFPalg} that $\mathcal J[h]$ is defined by~\eqref{TwistedFisher}. Then it satisfies the differential inequality
\[
\frac d{dt}\mathcal J[h(t,\cdot,\cdot)]\le-\,\mathcal J[h(t,\cdot,\cdot)]\,,
\]
{}from which we deduce that
\[
\mathcal J[h(t,\cdot,\cdot)]\le\mathcal J[h(0,\cdot,\cdot)]\,e^{-t}\quad\forall\,t\ge0\,.
\]
Using~\eqref{Ineq:E-EP} with $d\gamma=\mu\,dx\,dv$, $\lambda=1$ and $\varphi=\varphi_p$ for any $p\in[1,2]$ (also see Remark~\ref{Rem:OptGaussian}), we obtain that
\[
\mathcal E[h(t,\cdot,\cdot)]\le\mathcal J[h_0]\,e^{-t}\quad\forall\,t\ge0
\]
if $h$ is the solution of~\eqref{VFP3} with initial datum $h_0$.

The optimality of the rate is established by considering an initial datum which is a decentred stationary solution. With the notations of Section~\ref{Sec:Intro}, let
\[
f_0(x,v)=f_\star(x-x_0,v-v_0)\quad\forall\,(x,v)\in\R^d\times\R^d
\]
for some $(x_0,v_0)\neq(0,0)$. The reader is invited to check that
\begin{multline}\label{Ex:VFP}
f(t,x,v)=f_\star\big(x-x_\star(t),v-v_\star(t)\big)\\
\quad\mbox{with}\quad
\left\{\begin{array}{l}
x_\star(t)=\(\cos\big(\frac{\sqrt3}2\,t\big)\,x_0+\frac2{\sqrt3}\,\sin\big(\frac{\sqrt3}2\,t\big)\(v_0+\frac{x_0}2\)\)e^{-\frac t2}\,,\\
v_\star(t)=\(-\frac{\sqrt3}2\,\sin\big(\frac{\sqrt3}2\,t\big)\(x_0+\frac{v_0}2\)+\cos\big(\frac{\sqrt3}2\,t\big)\,v_0\)e^{-\frac t2}\,,
\end{array}\right.
\end{multline}
solves~\eqref{VFP1}. Now let us compute the entropy as $t\to+\infty$: with $g=f/f_\star$ and $\varphi=\varphi_p$, we obtain that, as $t\to+\infty$,
\begin{multline*}
\mathcal E[g(t,\cdot,\cdot)]=\iint_{\R^d\times\R^d}\varphi_p(g)\,d\mu=\frac p2\iint_{\R^d\times\R^d}|g-1|^2\,d\mu\,(1+o(1))\\
=\frac p2\(|x_\star(t)|^2+|v_\star(t)|^2\)(1+o(1))=O\(e^{-t}\)\,.
\end{multline*}
This proves that the rate $e^{-t}$ of Proposition~\ref{Prop:AE} is optimal and completes the proof.~ \qed

\medskip Compared to the proof of Proposition~\ref{Prop:AE}, a refined estimate can be obtained by observing that, in the computation of $\frac d{dt}\iri{|\nabla_vh|^2}$ and $\frac d{dt}\iri{|\nabla_xh|^2}$, we have
\begin{eqnarray*}
&&\|\H{vv}\|^2-\kappa\,\H{vv}:\M{vv}+2\,\kappa\,\|\M{vv}\|^2\ge0\,,\\
&&\|\H{xv}\|^2-\kappa\,\H{xv}:\M{xv}+2\,\kappa\,\|\M{xv}\|^2\ge0\,,
\end{eqnarray*}
with $\kappa=8\,(2-p)/p$. Let us define
\begin{multline*}
\mathsf a:=e^t\iri{|\nabla_vh|^2}\,,\quad\mathsf b:=e^t\iri{\nabla_vh\cdot\nabla_xh}\,,\quad\mathsf c:=e^t\iri{|\nabla_xh|^2}\,,\\
\mbox{and}\quad\mathsf j:=\mathsf a+\mathsf b+\mathsf c\,.
\end{multline*}
We deduce from~\eqref{Est1},~\eqref{Est2} and~\eqref{Est3} that
\[
\frac{d\kern 1pt\mathsf a}{dt}\le\mathsf a-\,2\,(\mathsf j-\mathsf c)\,,\quad\frac{d\kern 1pt\mathsf c}{dt}\le2\,(\mathsf j-\mathsf a)-\,\mathsf c\quad\mbox{and}\quad\frac{d\kern 1pt\mathsf j}{dt}\le0
\]
while we know by definition of $\mathsf a$, $\mathsf b$ and $\mathsf c$ and by the Cauchy-Schwarz estimate that
\[
\mathsf a\ge0\,,\quad\mathsf c\ge0\quad\mbox{and}\quad\mathsf b^2\le\mathsf a\,\mathsf c\,.
\]
In terms of $\mathsf a$ and $\mathsf c$, the inequality $\mathsf b^2=(\mathsf a+\mathsf c-\mathsf j)^2\le\mathsf a\,\mathsf c$ means that the problem is constrained to the interior of an ellipse, and that $\mathsf a=0$ if and only if $\mathsf c=\mathsf j$: see Fig.~\ref{Fig:F2}. Finally, let us observe that we have the following property.
\begin{lemma}\label{Lem:VFPannulation} Assume that $p\in[1,2]$, $\potential(x)=|x|^2/2$ and let $h$ be a solution of~\eqref{VFP3} with initial datum $h_0\in\mathrm L^1\cap\mathrm L^p(\R^d,d\gamma)$. With the above notations, if for some $t_0>0$, $\mathsf a(t_0)=0$ and $\mathsf j(t_0)\neq0$, then for any $t>t_0$ with $t-t_0$ small enough, we have $\mathsf a(t)>0$.\end{lemma}
\begin{proof} From the equivalence of~\eqref{VFP1} and~\eqref{VFP3}, we know that $h$ is smooth because of the expression of Green's function. By definition of $\mathsf b$ and $\mathsf j$, we have that $\mathsf b(t_0)=0$ and $\mathsf c(t_0)=\mathsf j(t_0)>0$. Since $\mathsf a(t_0)=0$ means that $h$ does not depend on $v$, we know that $\frac{d\kern 1pt\mathsf j}{dt}(t_0)=\mathsf j(t_0)>0$, hence proving that $\mathsf a(t)>0$ for $t-t_0>0$, small, because of the condition $\mathsf b^2\le\mathsf a\,\mathsf c$ and $\frac{d\kern 1pt\mathsf c}{dt}\le0$, which means that $t\mapsto(\mathsf a(t),\mathsf c(t))$ is constrained to the interior of the ellipse of Fig.~\ref{Fig:F2}.\end{proof}

\begin{figure}[ht]
\begin{center}
\includegraphics{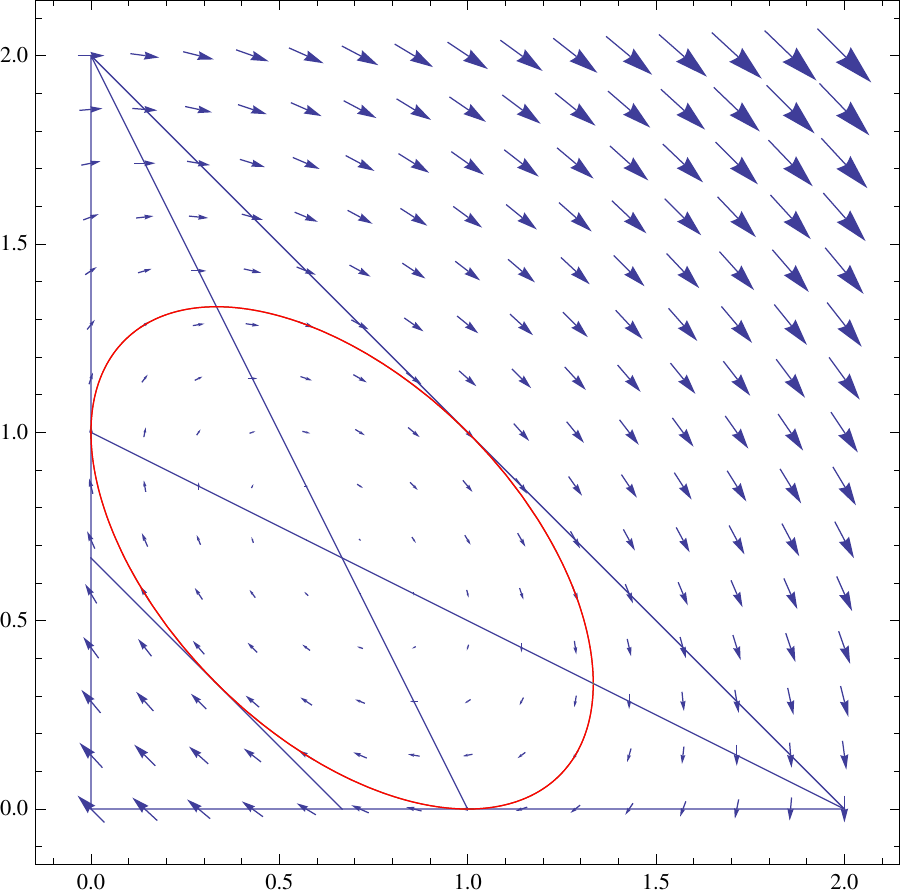}
\caption{\label{Fig:F2} Plot of the vector field associated with the ODEs $\frac{d\kern 1pt\mathsf a}{dt}=\mathsf a-\,2\,(\mathsf j-\mathsf c)$ and $\frac{d\kern 1pt\mathsf c}{dt}=2\,(\mathsf j-\mathsf a)-\,\mathsf c$. The coordinates are $\mathsf a/\mathsf j$ (horizontal axis) and $\mathsf c/\mathsf j$ (vertical axis). The two straight lines intersecting at the center of the ellipse are defined by $2\,(\mathsf j-\mathsf a)-\mathsf c=0$ and $\mathsf a-2\,\mathsf j+2\,\mathsf c=0$.}
\end{center}
\end{figure}

\subsection{Proof of Theorem \texorpdfstring{\ref{Thm:Main}}{1.1}}

Let us consider the \emph{Fisher information} functional as defined in~\eqref{TwistedFisher2}. A computation shows that
\[
-\frac12\,\frac d{dt}\mathcal J_{\lambda(t)}[h(t,\cdot)]=X^\perp\cdot\mathfrak M_1\,X-\frac12\,\lambda'(t)\,X^\perp\cdot\(
\begin{array}{cc}
0&1\\
1&0\\
\end{array}
\)\,X+Y^\perp\cdot\mathfrak M_2\,Y
\]
where $\mathfrak M_0$, $\mathfrak M_1$ and $\mathfrak M_2$ are defined as before, with $\munu=1$, and $X=\(\nabla_vh,\nabla_xh\)$, $Y=\(\H{vv},\H{xv},\M{vv},\M{xv}\)$. We take of course $\lambda=\lambda(t)$. We know that
\[
Y^\perp\cdot\mathfrak M_2\,Y\ge\lambda_1(p,\lambda)\,|Y|^2
\]
for some $\lambda_1(p,\lambda)$ such that $\lambda_1(p,1/2)=\frac14\(2\,\kappa+1-\,\sqrt{5\,\kappa^2-4\,\kappa+1}\)>0$ if $p\in(1,2)$, and $\kappa=8\,(2-p)/p$. For any $p\in(1,2)$, by continuity we know that $\lambda_1(p,\lambda)>0$ if $\lambda-1/2>0$ is taken small enough. From $|Y|^2\ge\|\M{vv}\|^2$ and, by Cauchy-Schwarz,
\[
\(\iri{|\nabla_vh|^2}\)^2\le\iri{h^2}\iri{\|\M{vv}\|^2}\le c_0\,\iri{\|\M{vv}\|^2}
\]
where $c_0:=1+(p-1)\,\mathcal E[h_0^{2/p}]$, we obtain
\[
-\frac12\,\frac d{dt}\mathcal J_{\lambda(t)}[h(t,\cdot)]\ge X^\perp\cdot\mathfrak M_1\,X+\frac12\,\lambda'(t)\,X^\perp\cdot\mathfrak M_0\,X+\varepsilon\,X^\perp\cdot\mathfrak M_3\,X
\]
with $\varepsilon=\lambda_1(p,\lambda)\,c_0^{-1}\iri{|\nabla_vh|^2}$ and
$\displaystyle \mathfrak M_3=\(
\begin{array}{cc}
1&0\\
0&0\\
\end{array}
\)\otimes\,\mathrm{Id}_{\R^d}$. We recall that $\mathsf a$ is defined by $\mathsf a=e^t\,\iri{|\nabla_vh|^2}$ is positive except for isolated values of $t>0$. Our goal is to find $\lambda(t)$ and $\rho(t)>1/2$ such that
\[
X^\perp\cdot\mathfrak M_1\,X-\frac12\,\lambda'(t)\,X^\perp\cdot\(
\begin{array}{cc}
0&1\\
1&0\\
\end{array}
\)\,X+\varepsilon\,X^\perp\cdot\mathfrak M_3\,X\ge\rho(t)\,X^\perp\cdot\mathfrak M_0\,X
\]
for any $X\in\R^{2d}$.

To establish the existence of $\rho>1/2$ a.e., we proceed in several steps.\\
$\bullet$ If $\mathsf a\ge\mathsf a_\star$ for some constant $\mathsf a_\star>0$, then we define $\varepsilon(t)=\nu\,e^{-t}$ with $\nu=\lambda_1(p,\lambda)\,c_0^{-1}\,\mathsf a_\star$, $\lambda(t)=(1+\varepsilon(t))/2$ and $\rho(t)=\frac12\,(1+\nu/(\nu+3\,e^t))$. The same estimate holds on any subinterval of $\R^+$.\\
$\bullet$ If $\mathsf a(t_0)=0$ for some $t_0\ge0$, then in a neighborhood of $(t_0)_+$, we can solve
\[
\frac{d\lambda}{dt}=\nu\,\varepsilon(t)\,,\quad\lambda(t_0)=\frac12\,.
\]
An eigenvalue computation shows that
\[
\mathfrak M_1+\frac12\,\nu\,\varepsilon\,\mathfrak M_0+\varepsilon\,\mathfrak M_3\ge\zeta(\varepsilon,\lambda,\nu)\,\mathfrak M_0
\]
with
\[
\zeta\(0,\tfrac12,\nu\)=\tfrac12\,,\quad\frac{\partial\zeta}{\partial\varepsilon}\(0,\tfrac12,\nu\)=\frac{2+\sqrt3-2\,\nu}3\,,\quad\frac{\partial\zeta}{\partial\lambda}\(0,\tfrac12,\nu\)=-\frac2{\sqrt3}\,.
\]
We choose an arbitrary $\nu\in(0,1+\sqrt3/2)$. Since $0<\lambda(t)-1/2=o(\varepsilon(t))$ for $t-t_0>0$, small enough, this guarantees that $\rho(t)=\zeta\(\varepsilon(t),\lambda(t),\nu\)$ satisfies $\rho(t)>1/2$ on a neighborhood of $(t_0)_+$.\\
$\bullet$ If $\mathsf\zeta(t_0)=0$ for some $t_0>0$, then in a neighborhood of $(t_0)_-$, we proceed as above with some $\nu<0$.\\
$\bullet$ If $(t_n)_{n\in\N}$ is the increasing sequence of points such that $\mathsf a(t_n)=0$ and if $\mathsf a(t)>0$ for any $t\in\R^+$ such that $t\neq t_n$ for any $n\in\N$, we can choose a constant $\mathsf a_\star$, small enough, on any interval $(t_n,t_{n+1})$ and glue the above solutions to obtain a function $\rho(t)>1/2$ on $(0,t_0)$ and $\cup_{n\in\N}(t_n,t_{n+1})$.
It is an open question to decide if there is an increasing sequence, finite or infinite, of times $t_n$ such that $\mathsf a(t_n)=0$, or if $\mathsf a(t)$ is positive for any $t>0$. We can of course impose that $\mathsf a(t_0)=0$ at $t_0=0$ by taking an initial datum $h_0$ which does not depend on $v$. If such a sequence $(t_n)_{n\in\N}$ exists, then we know that $\lambda(t_n)=1/2$ so that we have the remarkable decay estimate
\[
\mathcal J_{\tfrac12}[h(t_{n+1},\cdot)]\le\mathcal J_{\tfrac12}[h(t_n,\cdot)]\,e^{-\,2\int_{t_n}^{t_{n+1}}\rho(s)\,ds}<\mathcal J_{\tfrac12}[h(t_n,\cdot)]\,e^{-(t_{n+1}-t_n)}
\]
for any $p\in(1,2)$. As far as $\mathsf a$ is concerned, we expect that it has some oscillatory behaviour as indicated by the vector field in Fig.~\ref{Fig:F2}, but since terms involving $Y$ are neglected, this is so far formal. In any case, we can choose $\lambda(t)$ such that $\lim_{t\to+\infty}\lambda(t)=1/2$. This concludes the proof of Theorem~\ref{Thm:Main}.\nc\hfill\ \qed

\subsection{Concluding remarks}\label{Sec:conclusion}

Even if the global rate cannot be improved because it is determined by the large time asymptotics, at any finite time the instantaneous rate of decay is strictly higher in the case of the diffusions studied in Sections~\ref{Sec:EEP}-\ref{Sec:iEEP}, or at least higher at almost any time in the case of the kinetic equation, according to Theorem~\ref{Thm:Main}.

As $t\to+\infty$, Theorem~\ref{Thm:Main} provides us with an improved estimate of the leading order term. The exponential decay rate cannot be improved as shown by~\eqref{Ex:VFP}, but we prove that there is a constant less than $1$ to be taken into account. This observation is reminiscent of what happens for nonlinear diffusions of porous medium or fast diffusion type, which goes as follows. When looking at the relative entropy with respect to the \emph{best matching} (in the sense of relative entropy) profiles, it turns out that there is a delay $\tau$ compared to the relative entropy with respect to a fixed Barenblatt profile. As a result, we obtain a multiplicative factor $e^{-\tau}$ corresponding to an improved estimate in an asymptotic expansion as $t\to+\infty$~\cite{1501}. We have a similar property when we study the large time behavior of the solutions of~\eqref{VFP2} using a $\varphi_p$-entropy for any given $p\in(1,2)$.

The key estimate of Theorem~\ref{Thm:Main} asserts that
\[
\frac d{dt}\mathcal J_{\lambda(t)}[h(t,\cdot)]\le-2\,\rho(t)\,\mathcal J_{\lambda(t)}[h(t,\cdot)]\le-\mathcal J_{\lambda(t)}[h(t,\cdot)]
\]
where the last inequality is strict for almost any value of $t\ge0$ (unless $h$ is a stationary solution). Now, let us consider the large time asymptotics and define
\[
\tau:=\lim_{t\to+\infty}\left(2\int_0^t\rho(s)\,ds-t\right)\,.
\]
We cannot expect that $\tau=+\infty$ for any initial datum but at least show that $\tau$ is positive (unless $h$ is a stationary solution), so that for large values of $t$ we have
\be{FinalAsymptotics}
\mathcal J_{1/2}[h(t,\cdot)]\lesssim e^{-\tau}\,\mathcal J_{1/2}[h_0]\,e^{-t}\,.
\ee
For instance, in case of~\eqref{Ex:VFP}, one can prove that $\rho(t)-1/2$ is of the order of $e^{-t}$ and $\tau$ is finite. With $e^{-\tau}<1$,~\eqref{FinalAsymptotics} is anyway a strict improvement of the usual estimate as $t\to+\infty$.

The improvement of Theorem~\ref{Thm:Main} is obtained only for \emph{almost any} time: according to Lemma~\ref{Lem:VFPannulation}, the optimal decay rate could eventually be realized at an increasing sequence of times $t_n\nearrow+\,\infty$, but the solution will then deviate and temporarily regain a faster decay rate. Qualitatively, this comes from the oscillations in the phase space corresponding to the ODE associated with the vector field shown in Fig.~\ref{Fig:F2}. Such a pattern is consistent with what is known of the rates measured by \emph{hypocoercive methods} in kinetic equations.

\section*{Acknowledgments}
This work has been partially supported by the Projects Kibord and EFI (J.D.) of the French National Research Agency (ANR). The first author (J.D.) thanks J.-P.~Bartier and B.~Nazaret for fruitful discussions on, respectively, the tensorization properties of the $\varphi$-entropies and various computations based on the Bakry-Emery method, which took place at the occasion of a course taught by C.~Mouhot on hypocoercivity methods. Both authors thank an anonymous referee for interesting comments and suggestion which lead to improvements of this article.\\
\noindent{\scriptsize\copyright\,2018 by the authors. This paper may be reproduced, in its entirety, for non-commercial purposes.}

\bibliographystyle{siam}\small
\bibliography{References}


\end{document}